\documentclass{amsart}

\usepackage{amsfonts, amsmath, amssymb, amsthm, amsxtra, latexsym, hyperref}
\usepackage{xypic}

\theoremstyle{plain}

\newtheorem{theorem}{Theorem}[section]
\newtheorem{proposition}[theorem]{Proposition}

\newtheorem{corollary}[theorem]{Corollary}

\theoremstyle{definition}

\newtheorem{remark}[theorem]{Remark}

\numberwithin{equation}{section}
\newcommand{\LCA}{{\mathcal M}_l(A)}
\newcommand{\RCA}{{\mathcal M}_r(A)}
\newcommand{\DCA}{{\mathcal M}(A)}

\newcommand{\Ker}{{\textup{Ker}\,}}

\begin{document}
%%%%%%%%%%%%%%%%%%%%%%%%%%%%%%%%%%%%%%%%%%%%%%%%%%%%%%%%%%%%%%%%%%%%%%%%%%%%%%%%%%%%%%%

%%%%%%%%%%%%%%   Top matter %%%%%%%%%%%%%%%%%%%%%%%%%%%%%%%%%%%%%%%%%%%%%%%%%%%%%%%%%%%

\title{Extending representations of normed algebras in Banach spaces}

             % Linebreaks by \\, leave blank if not needed

\translator{}
             % Linebreaks by \\, leave blank if not needed

\dedicatory{}
            % Linebreaks by \\, leave blank if not needed

\author{Sjoerd Dirksen}
\address{Sjoerd Dirksen, Delft Institute of Applied Mathematics, Delft University of Technology, P.O.\ Box 5031,
2600 GA Delft,
The Netherlands}
\email{s.dirksen@tudelft.nl}

\author{Marcel de Jeu}
\address{Marcel de Jeu, Mathematical Institute,
Leiden University, P.O.\ Box 9512, 2300 RA Leiden,
The Netherlands}
\email{mdejeu@math.leidenuniv.nl}

\author{Marten Wortel}
\address{Marten Wortel, Mathematical Institute,
Leiden University, P.O.\ Box 9512, 2300 RA Leiden,
The Netherlands}
\email{wortel@math.leidenuniv.nl}

\begin{abstract}
Let $X$ be a non-degenerate left Banach module over a normed algebra $A$ having a bounded approximate left identity. We show that, if $A$ is a left ideal of a larger algebra, then this representation can be extended to a representation of the larger algebra. Based on this result, we study in detail the existence and properties of representations of the various centralizer algebras of $A$ which are compatible with the original representation of $A$. As a special case we obtain that, if $A$ embeds as a topological algebra into the bounded operators on $X$, then the left centralizer algebra of $A$ embeds as a topological algebra as the left normalizer of the image, and the double centralizer algebra of $A$ embeds as a topological algebra as the normalizer of the image. We also consider ordered and involutive contexts, and cover the right-sided cases, which are not always the obvious analogues of the left-sided cases, in detail as well.
\end{abstract}

%\date{}



\maketitle

%%%%%%%%%%%%%%   Start of the text %%%%%%%%%%%%%%%%%%%%%%%%%%%%%%%%%%%%%%%%%%%%%%%%%%%%

\section{Introduction}\label{sec:introduction}

This paper is concerned with the possibility of extending a given Banach representation of an ideal of a normed algebra to the whole algebra, and also, given a Banach representation of a normed algebra, with the possibility of defining representations of the various centralizer algebras of that algebra which are compatible with the given representation. These two issues are strongly related.

Our interest in this problem arose from the study of covariant representations of Banach algebra dynamical systems. In \cite{DidJWo} a crossed product Banach algebra is constructed from a Banach algebra dynamical system and a collection of covariant Banach space representations thereof, and the question is, roughly speaking, whether all representations of the crossed product are integrated forms of the given covariant representations of the original dynamical system. For $C^*$-dynamical systems and Hilbert representations the answer is affirmative, see, e.g., \cite{Williams}. The standard method to establish the result in that case is to extend a given representation of the $C^*$-crossed product to its multiplier algebra, after which the sought covariant representation of the $C^*$-dynamical system can be found with the aid of the canonical maps of the group and the algebra into the multiplier algebra. In this method, critical use is made of the standard result that a non-degenerate Hilbert representation of a closed two-sided ideal of a $C^*$-algebra extends to the algebra. For our case we needed a similar result for Banach representations, and with an upper bound for the norm of the extensions. With future applications in representation theory in Banach lattices in mind we were also interested in such results which take ordering into account.

Somewhat to our surprise, we were unable to find the results as we needed them in the literature. In the presence of a bounded approximate left identity, the seminal paper on centralizer algebras \cite{JohnsonLMS} contains material on Hilbert representations of centralizer algebras which are compatible with the original representation of the algebra, see \cite[Section 9]{JohnsonLMS}, but for more general spaces completeness assumptions on the image are made, cf.\ \cite[Theorem~20]{JohnsonLMS}. Likewise, when it is proved in \cite[Theorem~2.9.51]{Dales} that, in the presence of a bounded two-sided approximate identity, a Banach bimodule structure extends to the double centralizer algebra, it is assumed that the algebra is a Banach algebra, in order to be able to use the Cohen-Hewitt factorization theorem \cite[p.\ 61]{BonsallDuncan}, \cite[Theorem~2.9.24]{Dales}; the same holds for \cite[p.\ 17]{JohnsonMemoir}. Such completeness assumptions are, however, not necessary. As it turns out, it is possible to develop a theory for compatible non-degenerate Banach representations of centralizer algebras without using the factorization theorem, assuming only that the original algebra is a normed algebra with a suitable approximate identity. The key idea is essentially already in the proof of \cite[Theorem~21]{JohnsonMemoir}, but it appears that this has not yet been exploited systematically. This is done in the present paper, which can perhaps be regarded as a representation-theoretical supplement to the general material on centralizer algebras collected in \cite{PalmerI}.

It deserves to be mentioned at this point that, whereas the results in this paper on the existence of compatible representations of centralizer algebras of a normed algebra can, by passing to its completion, easily be derived from their versions for the centralizer algebras of a Banach algebra, this is no longer the case for our results on embeddings of centralizer algebras. The reason is simply that the centralizer algebras of a normed algebra may be strictly smaller than those of its Banach algebra completion. Thus our consideration of normed algebras rather than Banach algebras does not only make manifest that the factorization theorem is not needed, but it also avoids being unnecessarily restrictive as to the scope of the results.

The results on compatible representations of centralizer algebras in this paper rest on a basic theorem concerning extending a Banach representation of an ideal of a normed algebra to the algebra itself, cf.~the first part of Theorem~\ref{thm:introthm}. In spite of its elementary proof and its general relevance, we have not been able to find a reference for this result. It has some bearing even on the well known case of $C^*$-algebras, where, for Hilbert representations, such a results is usually stated for closed ideals and proved using GNS-theory, cf.~\cite[Proposition~2.10.4]{Dixmier}, \cite[Theorem~5.5.1]{Murphy}. There is an alternative approach to be found for the $C^*$-case which uses an approximate identity and which is close to ours, cf.\ \cite[II.6.1.6]{Blackadar}, \cite[Lemma~I.9.14]{Davidson}, but still the results in these sources are formulated under hypotheses which are more stringent than necessary: actually, the ideal needs only to be a left ideal, not necessarily closed, with a bounded left approximate identity for itself. Since they find it necessary to include a proof, the authors of \cite{Blackadar} and \cite{Davidson} also seem to be unaware of a reference for such a general extension result.

As a an illustration of what is implied by the mere presence of a bounded left approximate identity, we include the following excerpt of Theorem~\ref{thm:basicthm} and Theorem~\ref{thm:leftembeddings} in this introduction. The notions figuring in it will be properly defined later on, to avoid all possible misunderstanding, but they are the obvious ones. The representations in the formulation are norm continuous homomorphisms from the normed algebra in question to the algebra of bounded operators. Note that it is not assumed that the algebra is complete, nor that the ideal is closed.

\begin{theorem}\label{thm:introthm}
Let $A$ be a normed algebra, and let $X$ be a Banach space, with algebra of bounded operators $\mathcal B(X)$.
\begin{enumerate}
 \item If $J$ is a left ideal of $A$ containing a bounded approximate left identity for itself, and if $\pi:J\to\mathcal B(X)$ is a non-degenerate representation, then $\pi$ extends uniquely to a representation of $A$. If, in addition, $A$ is an ordered algebra, $X$ is an ordered vector space with a closed positive cone, $J$ contains a positive approximate left identity for itself, and $\pi$ is positive, then the extended representation is positive. Alternatively, if, in addition, $A$ has an involution which leaves $J$ invariant, $X$ is a Hilbert space, and $\pi$ is involutive, then the extended representation is involutive.
\item Suppose $A$ has a bounded approximate left identity, and $\pi:A\to\mathcal B(X)$ is a non-degenerate faithful representation which is an isomorphism of topological algebras between $A$ and $\pi(A)$. Then, as abstract algebras, $A$ and its double centralizer algebra $\DCA$ both embed canonically into the left centralizer algebra $\LCA$ of $A$. After identification one has $A\subset\DCA\subset\LCA$, and $\pi$ extends uniquely to a representation $\overline\pi:\LCA\to\mathcal B(X)$. Moreover, $\overline\pi$ is an isomorphism of topological algebras between $\LCA$ and the left normalizer of $\pi(A)$ in $\mathcal B(X)$, and its restriction to $\DCA$ (with its own norm) is an isomorphism of topological algebras between $\DCA$ and the normalizer of $\pi (A)$ in $\mathcal B(X)$. If, in addition, $A$ is an ordered algebra and has a positive bounded approximate left identity, $X$ is an ordered vector space with a closed positive cone, and $\pi$ is an isomorphism of ordered algebras, then the isomorphisms for $\LCA$ and $\DCA$ are isomorphisms of ordered algebras. Alternatively, if, in addition, $A$ has a bounded involution, $X$ is a Hilbert space and $\pi$ is involutive, then $\DCA$ has a bounded involution, and the isomorphism for $\DCA$ is an isomorphism of involutive algebras.
\end{enumerate}
\end{theorem}

The relation of the second part of Theorem~\ref{thm:introthm} with known results about centralizer algebras of the algebra of compact operators and about double centralizer algebras of $C^*$-algebras is discussed in Remarks~\ref{rem:isometricembedding} and~\ref{rem:compactoperators}.

For general $A$, it need not be the case that the canonical homomorphism of $A$ into $\LCA$ is injective, so that one cannot properly speak about extending representations from $A$ to $\LCA$ as in the second part of Theorem~\ref{thm:introthm}, and the statements then need to be phrased in terms of commutative diagrams expressing the compatibility of the original representation of the algebra and the representation of the centralizer algebra. The theorems in Section~\ref{sec:modulesovercentralizeralgebrasgeneralcase} contain such results, including upper bounds for the norms of the various maps. These results are valid for non-degenerate modules which are not necessarily faithful. In fact, the existence of a faithful non-degenerate module is equivalent to the injectivity of the natural homomorphism of $A$ into $\LCA$, see Proposition~\ref{prop:injectivity}.

When formulating our results, we have attempted to give as precise statements as possible, under minimal hypotheses. Following the statement for general Banach spaces, we have systematically also covered the ordered and Hilbert contexts. All this detailed information makes the statements rather long, but this seemed unavoidable.

We have also covered the case of non-degenerate right modules over a normed algebra with a bounded right approximate identity. We emphasize that one does \emph{not} obtain the results for the right-sided case from the left-sided case by simply replacing ``left'' with ``right'', and ``homomorphism'' with ``anti-homomorphism''. As an example, the canonical image of $A$ in $\LCA$ is a left ideal, but the canonical image in the right centralizer algebra $\RCA$ is not, in general, a right ideal: it is a left ideal. Furthermore, a left module over $A$ becomes a left module over both $\LCA$ and $\DCA$, but a right module over $A$ becomes a left (not: right) module over $\RCA$ and a right module over $\DCA$. To continue, the right-sided analogue of the embedding result of $\LCA$ and $\DCA$ in Theorem~\ref{thm:introthm} is that an anti-embedding of $A$ into $\mathcal B(X)$ yields an embedding (not: anti-embedding) of $\RCA$ as the left (not: right) normalizer of the image, and an anti-embedding of $\DCA$ as the normalizer of the image. The ``obvious'' adaptations of the results for the left-sided case to the right-sided case are therefore not the correct ones, and, although it adds to the length of the paper, it is for this reason that we felt it would be a disservice to the reader not to include the precise statements for the right-sided case in full. Thus Theorems~\ref{thm:leftcentralizers} and \ref{thm:leftembeddings}, and Corollary~\ref{cor:leftisomorphiccopies}, are concerned with the left-sided case, and Theorems~\ref{thm:rightcentralizers} and \ref{thm:rightembeddings}, and Corollary~\ref{cor:rightisomorphiccopies}, are concerned with the right-sided case. Theorems~\ref{thm:basicthm} and \ref{thm:doublecentralizers}, and Proposition~\ref{prop:injectivity}, are concerned with both cases. Naturally, the proofs for the right-sided case have been omitted, as they are completely similar to the left-sided case.

Although the lack of symmetry between the left-sided and the right-sided case may come as a surprise, there is an underlying reason for it: the standard terminology for algebras of linear maps has a left bias. One almost always --- and certainly always in this paper --- considers $A$ to be a left module over $\mathcal B(A)$, and $X$ to be a left module over $\mathcal B(X)$. This asymmetry, baked into the standard terminology, is what causes the ``discrepancies" later on. If, for an algebra with a bounded right approximate identity and right $A$-modules, one would use the opposite algebras of $\mathcal B(A)$ and $\mathcal B(X)$, then the symmetry in the statements would be restored. We felt, however, that using such formulation would be counterproductive. The whole phenomenon becomes perhaps most obvious in Corollary~\ref{cor:leftisometriccopies}, where tradition almost seems to oppose the mere idea of formulating a right-sided version. The authors, at least, were in this case content with only the left-sided result.

\medskip

This paper is organized as follows.

In Section~\ref{sec:basicterminologyandremarks} we introduce the basic terminology. This is standard, but including it makes the paper self-contained and also gives the opportunity to be precise about conventions concerning unitality, etc. We also include some remarks on a largest non-degenerate submodule and preservation of the set of invariant closed subspaces and of the set of intertwining operators.

Section~\ref{sec:extending from ideals} contains the basic result about extending a representation from an ideal to the algebra.

Section~\ref{sec:modulesovercentralizeralgebrasgeneralcase} starts with collecting some material on centralizer algebras, and then proceeds, in the general setting of non-degenerate Banach modules, to develop the results about the existence and properties of representations of these centralizers algebras which are compatible with representations of the original algebra.

In Section~\ref{sec:modulesovercentralizeralgebrasfaithfulcase} the results of Section~\ref{sec:modulesovercentralizeralgebrasgeneralcase} are strengthened when the module is faithful. If the algebra embeds, then so do the appropriate centralizer algebras.

\section{Basic terminology and preliminary remarks}\label{sec:basicterminologyandremarks}

We start by recalling some standard terminology and introducing notation.

\medskip

Assume that $A$ is a normed algebra over the field $\mathbb F$, where $\mathbb F$ is either $\mathbb R$ or $\mathbb C$. We do not require that $A$ is unital, nor that a possible identity element has norm $1$, but only that the norm is submultiplicative.

Suppose that $X$ is a normed space over $\mathbb F$ and that $\pi_l: A\to \mathcal B(X)$ is a bounded algebra homomorphism from $A$ into the algebra $\mathcal B(X)$ of bounded linear operators on $X$, thus providing $X$ with the structure of a normed left $A$-module. We do not assume that $\pi_l$ is unital if $A$ has an identity element. If the span of the elements $\pi_l(a)x$, for $a\in A$ and $x \in X$, is dense in $X$, then $X$ is said to be a non-degenerate normed left $A$-module, the homomorphism $\pi_l$ being understood. Similarly, if $\pi_r: A\to \mathcal B(X)$ is a bounded algebra anti-homomorphism, so that $X$ is a normed right $A$-module, and if the span of the elements $\pi_r(a)x$, for $a\in A$ and $x \in X$, is dense in $X$, then $X$ is said to be a non-degenerate normed right $A$-module. If $\pi_l$ and $\pi_r$ are a bounded algebra homomorphism, resp.\ a bounded algebra anti-homomorphism, such that $\pi_l(a_1)$ and $\pi_r(a_2)$ commute, for all $a_1, \,a_2\in A$, then $X$ is a normed $A$-bimodule, which is called a non-degenerate normed $A$-bimodule if the span of the elements $\pi_l(a_1)\pi_r(a_2)x$, for $x\in X$ and $a_1,a_2\in A$, is dense in $X$. The latter density is equivalent to $X$ being both a non-degenerate normed left $A$-module and a non-degenerate normed right $A$-module.

If $(T_i)_{i\in I}$ is a net in $\mathcal B(X)$ which converges in the strong operator topology to $T\in\mathcal B(X)$, i.e., if $\lim_i T_i x=Tx$ for all $x\in X$, then we will write $T=\textup{SOT-}\lim_i T_i$.

If $m>0$, then an $m$-bounded approximate left identity for $A$ is a net $(e_i)_{i\in I}$ in $A$, such that $\Vert e_i\Vert\leq m$, for all $i\in I$, and $\lim_i\Vert e_i a-a\Vert=0$, for all $a$ in $A$. Similarly one defines an $m$-bounded approximate right identity and an $m$-bounded two-sided approximate identity and one has obvious notions of bounded left, right, and two-sided approximate identities.

If $V$ is a vector space over $\mathbb F$, then a cone in $V$ is a non-empty subset $C$ such that $\lambda_1 c_1+\lambda_2 c_2\in C$ whenever $\lambda_1,\lambda_2\geq 0$ and $c_1,c_2\in W$. Declaring that $x\geq y$ whenever $x-y\in C$ introduces an ordering in $V$ with the usual properties. Note that, as in, e.g., \cite{Jameson}, we do not assume the properness $C\cap(-C)=\{0\}$ of the positive cone $C$, so that the relation $\geq$ need not be anti-symmetric. A map $\mu:V_1\to V_2$ between two ordered vector spaces is called positive whenever, for all $v_1\in V_1$, $v_1\geq 0$, implies $\mu(v_1)\geq 0$.

An ordered algebra is an algebra which is an ordered vector space with the additional property that, for all $a_1,a_2\in A$, $a_1,a_2\geq 0$ implies $a_1 a_2\geq 0$. If $X$ is an ordered normed space then the set of positive operators in $\mathcal B(X)$ is a cone, so that $\mathcal B(X)$ becomes an ordered normed algebra. If $A$ is an ordered algebra and $X$ is an ordered normed space which is a normed left $A$-module via $\pi_l$, we say that it is an ordered normed left $A$-module if $\pi_l$ is positive, i.e., when positive elements of $A$ act as positive operators in $X$. The analogous right-sided notions is obvious.

If $\mathbb F=\mathbb C$, an involution on an algebra $A$ is a map ${^*}:A\to A$ which is a conjugate linear anti-homomorphism of order 2; if $\mathbb F=\mathbb R$, an involution is a linear anti-homomorphism of order 2. We note explicitly that, when $A$ is normed, an involution is not required to be bounded. If $A$ and $B$ are two involutive algebras, then a map $\mu: A\to B$ is called involutive if $\mu(a^*)=\mu(a)^*$, for all $a\in A$.

\begin{remark}\label{rem:nondegeneratesubmodule}
Suppose $X$ is a normed left $A$-module. If $A$ has a bounded approximate left identity, then there exists a largest non-degenerate normed left $A$-submodule $X_{\textup{nd}}$ of $X$. Indeed, let $X_{\textup{nd}}$ be the closed linear span of the elements $\pi_l(a)x$, for $a$ in $A$ and $x$ in $X$. Surely any non-degenerate normed left $A$-submodule is contained in $X_{\textup{nd}}$. Moreover,  if $(e_i)_{i\in I}$ is a bounded approximate left identity in $A$, then $\pi_l(e_i)_{i\in I}$ is a norm bounded subset of $\mathcal B(X)$ and using this one sees easily that $\textup{SOT-}\lim_i \pi_l(e_i)\upharpoonright_{X_{\textup{nd}}}=\textup{id}_{X_{\textup{nd}}}$. In particular, $X_{\textup{nd}}$ is a non-degenerate normed left $A$-submodule as required. There are obvious right-sided and two-sided versions  for this. In the results below we will repeatedly encounter the assumption that $X$ is a non-degenerate normed $A$-module and from the present discussion we see that we can always pass from $X$ to the largest left, right or two-sided submodule satisfying this hypothesis.
\end{remark}

\begin{remark}\label{rem:samesubspacesetc}
In the subsequent sections, new representations will repeatedly be defined as SOT-limits of given ones. This implies that, under the new module structure, the set of closed invariant subspaces will remain unchanged. It also implies that the bounded intertwining operators between two such new representations will coincide with those for the two original representations. For reasons of space we make this general observation here once and for all, rather than add it on every separate occasion.
\end{remark}

\section{Extending from ideals}\label{sec:extending from ideals}

In this section we establish the basic theorem concerning extension of module structures initially defined for ideals. It is not necessary that the algebras are complete or that the ideals are closed, but we \emph{do} need that the spaces they act on are complete. In the next section we will apply the results in the context of centralizer algebras.

\newpage

\begin{theorem}\label{thm:basicthm}
Let $A$ be a normed algebra, and let $X$ be a Banach space.
\begin{enumerate}
\item If $J$ is a left ideal in $A$ containing an $m$-bounded approximate left identity for itself, and if the homomorphism $\pi_l:J\to X$ provides $X$ with the structure of a non-degenerate normed left $J$-module, then there exists a unique homomorphism $\overline \pi_l:A\to\mathcal B(X)$ extending $\pi_l$. This extension is, in fact, bounded with $\Vert\overline\pi_l\Vert\leq m \Vert\pi_l\Vert$, so that $X$ becomes a non-degenerate normed left $A$-module.
For $a\in A$, $\overline\pi_l(a)\in\mathcal B(X)$ is the unique bounded operator such that $\overline\pi_l(a)\pi_l(j)=\pi_l(aj)$, for all $j\in J$.  If $(e_i)_{i\in I}$ is any bounded approximate left identity in $J$ for itself, then $\textup{SOT-}\lim_i\pi_l(e_i)=\textup{id}_X$ and, for all $a\in A$, $\overline\pi_l(a)=\textup{SOT-}\lim_i \pi_l(ae_i)$. If $a\in A$ acts from the left on $J$ as the identity, then $\overline\pi_l(a)=\textup{id}_X$.

If, in addition, $A$ is an ordered algebra, $X$ is an ordered space with closed positive cone, $J$ contains a positive bounded approximate left identity for itself, and $\pi_l$ is positive, then $\overline\pi_l$ is positive.

Alternatively, if, in addition, $A$ has an involution which leaves $J$ invariant (so that $J$ is a two-sided ideal), $X$ is a Hilbert space, and $\pi_l$ is involutive, then $\overline\pi_l$ is involutive.

\item If $J$ is a right ideal in $A$ containing an $m$-bounded approximate right identity for itself, and if the anti-homomorphism $\pi_r:J\to X$ provides $X$ with the structure of a non-degenerate normed right $J$-module, then there exists a unique  anti-homomorphism $\overline \pi_r:A\to\mathcal B(X)$ extending $\pi_r$. This extension is, in fact, bounded with $\Vert\overline\pi_r\Vert\leq m \Vert\pi_r\Vert$, so that $X$ becomes a non-degenerate normed right $A$-module.
For $a\in A$, $\overline \pi_r(a)\in\mathcal B(X)$ is the unique bounded operator such that $\overline\pi_r(a)\pi_r(j)=\pi_r(ja)$, for all $j\in J$. If $(e_i)_{i\in I}$ is any bounded approximate right identity in $J$ for itself, then $\textup{SOT-}\lim_i\pi_r(e_i)=\textup{id}_X$ and, for all $a\in A$, $\overline\pi_r(a)=\textup{SOT-}\lim_i \pi_r(e_ia)$. If $a\in A$ acts from the right on $J$ as the identity, then $\overline\pi_r(a)=\textup{id}_X$.

If, in addition, $A$ is an ordered algebra, $X$ is an ordered space with closed positive cone, $J$ contains a positive bounded approximate right identity for itself, and $\pi_r$ is positive, then $\overline\pi_r$ is positive.

Alternatively, if, in addition $A$ has an involution which leaves $J$ invariant (so that $J$ is a two-sided ideal), $X$ is a Hilbert space, and $\pi_r$ is involutive, then $\overline\pi_r$ is involutive.

\item If $J$ is a two-sided ideal in $A$ containing both a bounded approximate left identity and a bounded approximate right identity for itself, and if the homomorphism $\pi_l:J\to\mathcal B(X)$ and the anti-homomorphism $\pi_r:J\to\mathcal B(X)$ provide $X$ with the structure of a non-degenerate normed $J$-bimodule, then the maps $\overline\pi_l$ and $\overline\pi_r$ from the first two parts make $X$ into a non-degenerate normed $A$-bimodule.
    \end{enumerate}
\end{theorem}

\begin{proof}
As to the first part, let $(e_i)_{i\in I}$ be an $m$-bounded approximate left identity in $J$ for itself. Fix $a\in A$. Now $\Vert\pi_l(ae_i)\Vert\leq m \Vert\pi_l\Vert \Vert a\Vert$, for all $i\in I$, and using this uniform bound a $3\epsilon$-argument easily implies that the set $\{x\in X : (\pi_l(ae_i)x)_{i\in I} \textup{ is a Cauchy net}\}$ is a closed linear subspace. Since $(e_i)_{i\in I}$ is an approximate left identity for $J$, this set clearly contains all elements of the form $\pi_l(j)y$ for $y\in X$ and $j\in J$, hence is equal to $X$ by the non-degeneracy of $X$ as a normed left $J$-module. This enables us to define, for all $x\in X$,
\begin{equation}\label{eq:basiceq}
\overline\pi_l(a)x=\lim_i \pi_l(ae_i)x.
\end{equation}
It is obvious that $\overline\pi_l(a)\in \mathcal B(X)$ and that $\Vert\overline\pi_l(a)\Vert\leq m \Vert\pi_l\Vert \Vert a\Vert$. Using once more the fact that $(e_i)_{i\in I}$ is an approximate left identity for $J$, as well as the non-degeneracy of $X$ as a normed left $J$-module, one sees that $\overline\pi_l$ extends $\pi_l$ and also that $\overline\pi_l(a)\pi_l(j)=\pi_l(aj)$, for $a\in A$ and $j\in J$. By the non-degeneracy as a normed left $J$-module, the latter equation determines $\overline\pi_l(a)$ uniquely as an element of $\mathcal B(X)$. Since, for $a,b\in A$ and $j\in J$, one has $\overline\pi_l(a)\overline\pi_l(b)\pi_l(j)=\overline\pi_l(a)\pi_l(bj)=\pi_l(abj)=\overline\pi_l(ab)\pi_l(j)$, we conclude from the non-degeneracy as a normed left $J$-module that $\overline\pi_l$ is a homomorphism. This establishes the statements in the first part of the proposition regarding general $X$.

In the involutive context, for $x,y\in H$ and $a\in A, j\in J$, we compute that
\begin{align*}\langle\overline\pi_l(a)\pi_l(j)x,y\rangle&=\langle \pi_l(aj)x,y\rangle=\langle x, \pi_l(aj)^*y\rangle=\langle x,\pi_l((aj)^*)y\rangle\\
&=\langle x,\pi_l(j^*a^*)y\rangle=\langle x,\overline\pi_l(j^*a^*)y\rangle=\langle x,\overline\pi_l(j^*)\overline\pi_l(a^*)y\rangle\\
&=\langle x,\pi_l(j^*)\overline\pi_l(a^*)y\rangle=\langle \pi_l(j)x,\overline\pi_l(a^*)y\rangle=\langle\overline\pi_l(a^*)^*\pi_l(j)x,y\rangle.
\end{align*}
Hence, by the non-degeneracy of $X$ as a normed left $J$-module again, $\overline\pi_l(a)=\overline\pi_l(a^*)^*$, so that $\overline\pi_l$ is involutive.

The statement in the ordered context is clear from \eqref{eq:basiceq} and the fact that the positive cone is closed.

The proof of the first part is now complete and the proof of the second part is similar. The third part follows from the first two parts since, for all $a_1\in A$, resp.\ for all $a_2\in A$, the operator $\overline\pi_l(a_1)$, resp.\ $\overline\pi_r(a_2)$, is an element of the strong operator closure of $\pi_l(J)$, resp.\ $\pi_r(J)$, and hence these operators commute.
\end{proof}

\begin{remark}\quad
\begin{enumerate}
\item Note that in the first and second part it is not required that the involution is bounded.
\item If $A$ as in Theorem~\ref{thm:basicthm} is a Banach algebra, then in the first part one does not have to require that $\overline\pi(a)$ is a bounded linear operator, because as a consequence of the Cohen-Hewitt factorization theorem \cite[p.\ 61]{BonsallDuncan}, \cite[Theorem~2.9.24]{Dales} the requirement $\overline\pi_l(a)\pi_l(j)=\pi_r(aj)$, for all $j\in J$, already determines $\overline\pi_r(a)$ as a map from $X$ into itself. Linearity and boundedness are then automatic. A similar remark applies to the second part.
\end{enumerate}
\end{remark}

\section{Module structures for centralizer algebras: general case}\label{sec:modulesovercentralizeralgebrasgeneralcase}

In this section we are concerned with the possibility of finding module structures for centralizer algebras of an algebra $A$ which are compatible with a given module structure for $A$. We cannot directly apply Theorem~\ref{thm:basicthm} because, although $A$ maps canonically onto an ideal in its various centralizer algebras, such maps need not be injective. However, for non-degenerate modules the initial (anti)-representations of $A$ does, in fact, descend to the images in the centralizer algebras, and subsequentely Theorem~\ref{thm:basicthm} can be applied to that situation.

\medskip

We start with the necessary preparations. Suppose $A$ is a normed algebra.

Let $\LCA=\{L\in\mathcal B(A) : L(ab)=L(a)b\textup{ for all }a,b\in A\}$ be the left centralizer algebra of $A$. It is sometimes called the right centralizer algebra, which is perhaps more logical since the operators in $\LCA$ commute with (i.e.: centralize) all right multiplications rather than the left ones, but we adhere to Johnson's choice of terminology in his seminal paper \cite{JohnsonLMS}. Likewise, $\RCA=\{R\in\mathcal B(A) : R(ab)=aR(b)\textup{ for all }a,b\in A\}$ is the right centralizer algebra of $A$. A pair $(L,R)$ with $L\in\LCA$ and $R\in\RCA$ is called a double centralizer if $aL(b)=R(a)b$ for all $a,b\in A$. Clearly $\LCA$ and $\RCA$ are unital closed subalgebras of $\mathcal B(A)$. Defining $(L_1,R_1)\cdot(L_2,R_2)=(L_1\circ L_2, R_2\circ R_1)$ makes $\DCA$ into a unital algebra over $\mathbb F$, which becomes a normed algebra if one puts $\Vert(L,R)\Vert=\max(\Vert L\Vert,\Vert R\Vert)$.

If\ ${^*}:A\to A$ is a bounded involution and $L\in\LCA$, then the map $L^*:A\to A$ defined by $L^*a=(L(a^*))^*$ is a right centralizer. This yields a bounded unital homomorphism ${^*}:\LCA\to\RCA$ (which is conjugate linear if $\mathbb F=\mathbb C$), inverse to the similarly defined bounded unital homomorphism ${^*}:\RCA\to\LCA$ (which is conjugate linear if $\mathbb F=\mathbb C$). Combining these yields a bounded involution ${^*}:\DCA\to\DCA$, defined as $(L,R)^*=(R^*,L^*)$. Thus $\DCA$ is a unital normed algebra with bounded involution.

If $A$ is an ordered algebra, then so are $\LCA$ and $\RCA$. Furthermore, $\DCA$ then also becomes an ordered algebra by defining $(L,R)\in\DCA$ to be positive if $L\geq 0$ and $R\geq 0$.

There is a canonical contractive homomorphism $\lambda: A\to\LCA$, defined by $\lambda(a)b=ab$, for $a,b\in A$. Since $L\circ\lambda(a)=\lambda(L(a))$, for $a\in A$ and $L\in\LCA$, $\lambda(A)$ is a left ideal in $\LCA$. If $A$ has an ($m$-bounded) approximate left identity $(e_i)_{i\in I}$, then $(\lambda(e_i))_{i\in I}$ is an ($m$-bounded) approximate left identity in $\lambda(A)$. If $A$ is an ordered algebra, then $\lambda$ is positive, and positive approximate left identities in $A$ yield positive approximate left identities in $\lambda(A)$.

Likewise, we have a canonical contractive anti-homomorphism $\rho:A\to\RCA$, defined by $\rho(a)b=ba$, for $a,b\in A$, and since $R\circ\rho(a)=\rho(R(a))$, for all $a\in A$ and $R\in\RCA$, the image $\rho(A)$ is a left ideal in $\RCA$. If $A$ has an ($m$-bounded) approximate right identity $(e_i)_{i\in I}$, then $(\rho(e_i))_{i\in I}$ is an ($m$-bounded) approximate left identity in $\rho(A)$. If $A$ is an ordered algebra, then $\rho$ is positive, and positive approximate right identities in $A$ yield positive approximate left identities in $\rho(A)$.

The map $\delta:A\to\DCA$ which is defined, for $a\in A$, by $\delta(a)=(\lambda(a),\rho(a))$, is a contractive homomorphism. For $a\in A$ and $(L,R)\in\DCA$ one computes that $(L,R)(\lambda(a),\rho(a))=(\lambda(L(a)),\rho((L(a)))$ and similarly  $(\lambda(a),\rho(a))(L,R)=(\lambda(R(a)),\rho(R(a)))$; hence the image $\delta(A)$ is a two-sided ideal in $\DCA$. If $A$ has an ($m$-bounded) approximate left, resp.\ right, identity $(e_i)_{i\in I}$, then $(\delta(e_i))_{i\in I}$ is an ($m$-bounded) left, resp.\ right, approximate identity in $\delta(A)$. If $A$ is an ordered algebra, then $\delta$ is positive, and positive left, resp.\ right, approximate identities in $A$ yield positive left, resp.\ right, approximate identities in $\delta(A)$. If $A$ has a bounded involution, then, for $a\in A$, one has $\lambda(a)^*=\rho(a^*)$ and $\rho(a)^*=\lambda(a^*)$. In that case $\delta: A\to\DCA$ is a contractive involutive homomorphism.

Retaining the $L$-part in $(L,R)$ gives a unital contractive homomorphism $\phi_l:\DCA\to\LCA$ which maps $\delta(A)$ onto $\lambda(A)$, and retaining the $R$-part yields a unital contractive anti-homomorphism $\phi_r:\DCA\to\RCA$ which maps $\delta(A)$ onto $\rho(A)$. If $A$ is an ordered algebra, then both $\phi_l:\DCA\to\LCA$ and $\phi_r:\DCA\to\RCA$ are positive.

Suppose that the homomorphism $\pi_l:A\to \mathcal B(X)$ provides the normed space $X$ with the structure of a non-degenerate normed left $A$-module. From the non-degeneracy it is clear that $\pi_l(\Ker\lambda)=0$, hence there is a unique map $\widetilde\pi_l:\lambda(A)\to\mathcal B(X)$ such that
\begin{equation}\label{diag:leftfirst}
\xymatrix
{
A \ar[d]_\lambda \ar[r]^{\pi_l}&\mathcal B(X)\\
\lambda(A)\ar[ru]_{\widetilde\pi_l}}
\end{equation}
is commutative; it is in fact a homomorphism. If $A$ has an $m$-bounded approximate left identity, then $\widetilde\pi_l$ is a bounded homomorphism making $X$ into a non-degenerate normed left $\lambda(A)$-module, and $\Vert\widetilde\pi_l\Vert\leq m\Vert\pi_l\Vert$. To see this, let $(e_i)_{i\in I}$ be such an $m$-bounded approximate left identity. We know from Theorem~\ref{thm:basicthm} that $\textup{SOT-}\lim_i \pi_l(e_i)=\textup{id}_X$. Hence, for $a\in A$ and $x\in X$, one has $\widetilde\pi_l (\lambda(a))x=\pi_l(a)x=\lim_i \pi_l(a)\pi_l(e_i)x=\lim_i\pi_l(ae_i)x=\lim_i\pi_l(\lambda(a)e_i)x$. Since we have $\Vert \pi_l(\lambda(a)e_i)x\Vert\leq \Vert\pi_l\Vert\Vert\lambda(a)\Vert\Vert e_i\Vert\Vert x\Vert \leq m\Vert\pi_l\Vert\Vert\lambda(a)\Vert\Vert x\Vert$, the conclusion follows. If $A$ is an ordered algebra which has a positive bounded approximate left identity, if $X$ is ordered with a closed positive cone and $\pi_l$ is positive, then $\widetilde\pi_l$ is positive. This follows from the equation $\widetilde\pi_l (\lambda(a))x=\lim_i\pi_l(\lambda(a)e_i)x$ derived above.

Likewise, if the anti-homomorphism $\pi_r:A\to\mathcal B(X)$ makes $X$ into a non-degenerate normed right $A$-module, then there is a unique map $\widetilde\pi_r:\rho(A)\to\mathcal B(X)$ such that
\begin{equation}\label{diag:rightfirst}
\xymatrix
{
A\ar[d]_\rho \ar[r]^{\pi_r}&\mathcal B(X)\\
\rho(A)\ar[ru]_{\widetilde\pi_r}}
\end{equation}
is commutative; it is in fact a homomorphism. If $A$ has an $m$-bounded approximate right identity, then $\widetilde\pi_r$ is a bounded homomorphism making $X$ into a non-degenerate normed left $\rho(A)$-module, and $\Vert\widetilde\pi_r\Vert\leq m\Vert\pi_r\Vert$. If $A$ is an ordered algebra which has a positive bounded approximate right identity, if $X$ is ordered with a closed positive cone and $\pi_r$ is positive, then $\widetilde\pi_r$ is positive.

\medskip

We are now in the position to apply Theorem~\ref{thm:basicthm} to centralizer algebras. The left-sided version is as follows.

\newpage

\begin{theorem}\label{thm:leftcentralizers}
Let $A$ be a normed algebra with an $m$-bounded approximate left identity, and let $X$ be a Banach space.

If $\pi_l:A\to\mathcal B(X)$ provides $X$ with the structure of a non-degenerate normed left $A$-module, then there exist a unique map $\widetilde\pi_l$ and a unique homomorphism $\overline\pi_l:\LCA\to\mathcal B(X)$ such that the diagram
\begin{equation}\label{diag:leftsecond}
\xymatrix
{A\ar[r]^{\pi_l}\ar[d]_{\lambda} & \mathcal B(X)\\
\lambda(A)\ar@{^{(}->}[r]_i \ar[ru]_{\widetilde{\pi}_l}&\LCA\ar[u]_{\overline{\pi}_l}}
\end{equation}
is commutative. All maps in the diagram are bounded homomorphisms, and $\overline\pi_l$ is unital. One has $\Vert\lambda\Vert\leq 1$, $\Vert i\Vert=1$ if $\lambda(A)\neq 0$, $\Vert\widetilde\pi_l\Vert\leq m\Vert\pi_l\Vert$, and $\Vert\overline\pi_l\Vert\leq m\Vert\pi_l\Vert$. In particular, $X$ becomes a non-degenerate normed left $\LCA$-module.

The image $\pi_l(A)$ is a left ideal in $\overline\pi_l(\LCA)$. In fact, if $L\in\LCA$ and $a\in A$, then $\overline\pi_l(L)\pi_l(a)=\pi_l(L(a))$.

If $(e_i)_{i\in I}$ is any bounded approximate left identity for $A$, then $\textup{SOT-}\lim_i\pi_l(e_i)=\textup{id}_X$, and, if $L\in\LCA$, then $\overline\pi_l(L)=\textup{SOT-}\lim_i \pi_l(L(e_i))$.

If, in addition, $A$ is an ordered algebra with a positive bounded approximate left identity, if $X$ is ordered with a closed positive cone, and if $\pi_l$ is positive, then all algebras in the diagram are ordered and all maps are positive.
\end{theorem}

\begin{proof}
We know from the discussion of centralizer algebras that $\lambda(A)$ is a left ideal in $\LCA$, and that $(\lambda(e_i))_{i\in I}$ is an ($m$-bounded)(positive) approximate left identity in $\lambda(A)$ if $(e_i)_{i\in I}$ is an ($m$-bounded) (positive) approximate left identity in $A$. Furthermore, in the results surrounding diagram~\eqref{diag:leftfirst} we have already observed that $\widetilde\pi_l:\lambda(A)\to\mathcal B(X)$ is the unique map making the upper triangle commutative, and that it is in fact a bounded homomorphism with $\Vert\widetilde\pi_l\Vert\leq m\Vert\pi_l\Vert$.

Hence the first part of Theorem~\ref{thm:basicthm} applies to this situation and it provides the unique homomorphism $\overline\pi_l$ making the lower triangle commutative. It also yields that, for any bounded approximate left identity $(e_i)_{i\in I}$ in $A$, and, for all $L\in\LCA$ and $x\in X$, $\overline\pi_l(L)x=\lim_i\widetilde\pi_l(L\circ\lambda(e_i))x=\lim_i\widetilde\pi_l(\lambda(L(e_i)))x=\lim_i\pi_l(L(e_i))x$. It also shows that $\overline\pi_l$ is bounded and that $\Vert\overline\pi_l\Vert\leq m\Vert\widetilde\pi_l\Vert\leq m^2\Vert\pi_r\Vert$, but this is not optimal: choosing an $m$-bounded approximate left identity $(e_i)_{i\in I}$ for $A$ one sees immediately from $\overline\pi_l(L)x=\lim_i \pi_l(L(e_i))x$ that, in fact, $\Vert\overline\pi_l\Vert\leq m\Vert\pi_l\Vert$. Since $\lambda(A)$ is a left ideal in $\LCA$, the same holds for the images $\overline\pi_l(\lambda(A))=\pi_l(A)$ and $\overline\pi_l(\LCA)$. In fact, for $L\in\LCA$ and $a\in A$, we have $\overline\pi_l(L)\pi_l(a)=\overline\pi_l(L)\overline\pi_l(\lambda(a))=\overline\pi_l(L \circ\lambda(a))=\overline\pi_l(\lambda(L(a)))=\pi_l(L(a))$.

The remaining statements are either clear or follow from the first part of Theorem~\ref{thm:basicthm}.
\end{proof}

\begin{remark}\label{rem:leftremark}
It is also true that $\overline\pi_l:\LCA\to\mathcal B(X)$ is the unique homomorphism making the square in diagram \eqref{diag:leftsecond} commutative. Indeed, for a homomorphism with this property one sees that, for $a\in A$, $x\in X$, and $L\in\LCA$, $\overline\pi_l(L)\pi_l(a)x=\overline\pi_l(L)\overline\pi_l(\lambda(a))x=\overline\pi_l(L\circ\lambda(a))x=\overline\pi_l(\lambda(L(a)))x=\pi_l(L(a))x$. Thus $\overline\pi_l(L)$ is uniquely determined, as a consequence of the non-degeneracy of $X$.
\end{remark}

The right-sided version of Theorem~\ref{thm:leftcentralizers} is not obtained by replacing left with right and homomorphism with anti-homomorphism. Instead, it reads as follows.

\begin{theorem}\label{thm:rightcentralizers}
Let $A$ be a normed algebra with an $m$-bounded approximate right identity, and let $X$ be a Banach space.

If $\pi_r:A\to\mathcal B(X)$ provides $X$ with the structure of a non-degenerate normed right $A$-module, then there exist a unique map $\widetilde\pi_r$ and a unique homomorphism $\overline\pi_r:\RCA\to\mathcal B(X)$ such that the diagram
\begin{equation}\label{diag:rightsecond}
\xymatrix
{A\ar[r]^{\pi_r}\ar[d]_{\rho} & \mathcal B(X)\\
\rho(A)\ar@{^{(}->}[r]_i \ar[ru]_{\widetilde{\pi}_r}&\RCA\ar[u]_{\overline{\pi}_r}}
\end{equation}
is commutative. Then $\pi_r$ and $\rho$ are bounded anti-homomorphisms, $\widetilde\pi_r$, $i$ and $\overline\pi_r$ are bounded homomorphisms, and $\overline\pi_r$ is unital. One has $\Vert\rho\Vert\leq 1$, $\Vert i\Vert=1$ if $\rho(A)\neq 0$, $\Vert\widetilde\pi_r\Vert\leq m\Vert\pi_r\Vert$, and $\Vert\overline\pi_r\Vert\leq m\Vert\pi_r\Vert$. In particular, $X$ becomes a non-degenerate normed left $\RCA$-module.

The image $\pi_r(A)$ is a left ideal in $\overline\pi_r(\RCA)$. In fact, if $R\in\RCA$ and $a\in A$, then $\overline\pi_r(R)\pi_r(a)=\pi_r(R(a))$.

If $(e_i)_{i\in I}$ is any bounded approximate right identity for $A$, then $\textup{SOT-}\lim_i\pi_r(e_i)\!=\textup{id}_X$, and, if $R\in\RCA$, then $\overline\pi_r(R)=\textup{SOT-}\lim_i \pi_r(R(e_i))$.

If, in addition, $A$ is an ordered algebra with a positive bounded approximate right identity, if $X$ is ordered with a closed positive cone, and if $\pi_r$ is positive, then all algebras in the diagram are ordered and all maps are positive.
\end{theorem}

\begin{remark}\label{rem:rightremark}
Analogously to Remark~\ref{rem:leftremark}, $\pi_r:\RCA\to\mathcal B(X)$ is the unique homomorphism making the square in diagram~\eqref{diag:rightsecond} commutative, as it must satisfy $\overline\pi_r(R)\pi_r(a)x=\pi_r(R(a))x$, for all $R\in\RCA$, $a\in A$, and $x\in X$.
\end{remark}

\begin{proof}
We know from the discussion of centralizer algebras that $\rho(A)$ is a left ideal in $\RCA$, and that $(\rho(e_i))_{i\in I}$ is an ($m$-bounded)(positive) approximate left identity in $\rho(A)$ if $(e_i)_{i\in I}$ is an ($m$-bounded) (positive) approximate right identity in $A$. Furthermore, in the results surrounding diagram~\eqref{diag:rightfirst} we have already observed that $\widetilde\pi_r:\rho(A)\to\mathcal B(X)$ is the unique map making the upper triangle commutative, and that it is in fact a bounded homomorphism with $\Vert\widetilde\pi_r\Vert\leq m\Vert\pi_r\Vert$.  Hence the first part of Theorem~\ref{thm:basicthm} applies to this situation, and the proof is completed as for Theorem~\ref{thm:leftcentralizers}.
\end{proof}

Next we turn to the associated $\DCA$-modules.

Recall the unital contractive homomorphism $\phi_l:\DCA\to\LCA$, defined by retaining the $L$-part, and the unital contractive anti-homomorphism $\phi_r:\DCA\to\RCA$, defined by retaining the $R$-part. As we have seen in Theorem~\ref{thm:leftcentralizers}, under suitable hypotheses a left $A$-module $X$ becomes a left $\LCA$-module through a homomorphism $\overline\pi_l:\LCA\to\mathcal B(X)$. Hence $X$ will become a left $\DCA$-module through the homomorphism $\overline\pi_l\circ\phi_l:\DCA\to\mathcal B(X)$. Likewise, from Theorem~\ref{thm:rightcentralizers} we see that under suitable hypotheses a right $A$-module will become a right $\DCA$-module through the anti-homomorphism $\overline\pi_r\circ\phi_r:\DCA\to\mathcal B(X)$ (recall that $\overline\pi_r:\RCA\to\mathcal B(X)$ is a homomorphism. The details are contained in the following result. The maps $\widetilde\pi_l$ and $\widetilde\pi_r$ figuring in the statements are again defined in the diagrams~\eqref{diag:leftfirst} and~\eqref{diag:rightfirst}, and $\widetilde\phi_l$, resp.\ $\widetilde\phi_r$, is the restriction of $\phi_l$, resp.\ $\phi_r$, to $\delta(A)$.

Since $\DCA$ is an involutive algebra when $A$ has a bounded involution, there are now also statements on Hilbert representations included, a new feature compared with Theorems~\ref{thm:leftcentralizers} and~\ref{thm:rightcentralizers}.

\begin{theorem}\label{thm:doublecentralizers}
Let $A$ be a normed algebra, and let $X$ be a Banach space.
\begin{enumerate}
\item
If  $A$ has an $m$-bounded approximate left identity, and $\pi_l:A\to\mathcal B(X)$ provides $X$ with the structure of a non-degenerate normed left $A$-module, then $\widetilde\pi_l\circ\widetilde\phi_l:\delta(A)\to\mathcal B(X)$ is the unique map and $\overline\pi_l\circ\phi_l:\DCA\to\mathcal B(X)$ is the unique homomorphism such that the diagram
\begin{equation}\label{diag:doubleleft}
\xymatrix
{A\ar[r]^{\pi_l}\ar[d]_{\delta} & \mathcal B(X)\\
\delta(A)\ar@{^{(}->}[r]_i \ar[ru]_{\widetilde{\pi}_l\circ\widetilde\phi_l}&\DCA\ar[u]_{\overline{\pi}_l\circ\phi_l}}
\end{equation}
is commutative. All maps in the diagram are bounded homomorphisms, and $\overline\pi_l\circ\phi_l$ is unital. One has $\Vert\delta\Vert\leq 1$, $\Vert i\Vert=1$ if $\delta(A)\neq 0$, $\Vert\widetilde\pi_l\circ\widetilde\phi_l\Vert\leq m\Vert\pi_l\Vert$, and $\Vert\overline\pi_l\circ\phi_l\Vert\leq m\Vert\pi_l\Vert$. In particular, $X$ becomes a non-degenerate normed left $\DCA$-module.

The image $\pi_l(A)$ is a two-sided ideal in $(\overline\pi_l\circ\phi_l)(\DCA)$. In fact, if $(L,R)\in\DCA$ and $a\in A$, then
$(\overline\pi_l\circ\phi_l)((L,R))\pi_l(a)=\overline\pi_l(L)\pi_l(a)=\pi_l(L(a))$, and $\pi_l(a)(\overline\pi_l\circ\phi_l)((L,R))=\pi_l(a)\overline\pi_l(L)=\pi_l(R(a))$.

If $(e_i)_{i\in I}$ is any bounded approximate left identity for $A$, and $(L,R)\in\DCA$, then $(\overline\pi_l\circ\phi_l)((L,R))=\textup{SOT-}\lim_i \pi_l(L(e_i))$.

If, in addition, $A$ is an ordered algebra with a positive bounded approximate left identity, if $X$ is ordered with a closed positive cone, and if $\pi_l$ is positive, then all algebras in the diagram are ordered and all maps are positive.

Alternatively, if, in addition, $A$ has a bounded involution, $X$ is a Hilbert space and $\pi_l$ is involutive, then all algebras in the diagram have a bounded involution and all maps are involutive.

\item

If $A$ has an $m$-bounded approximate right identity, and $\pi_r:A\to\mathcal B(X)$ provides $X$ with the structure of a non-degenerate normed right $A$-module, then $\widetilde\pi_r\circ\widetilde\phi_r:\delta(A)\to\mathcal B(X)$ is the unique map and $\overline\pi_r\circ\phi_r:\DCA\to\mathcal B(X)$ is the unique anti-homomorphism such that the diagram

\begin{equation}\label{diag:doubleright}
\xymatrix
{A\ar[r]^{\pi_r}\ar[d]_{\delta} & \mathcal B(X)\\
\delta(A)\ar@{^{(}->}[r]_i \ar[ru]_{\widetilde{\pi}_r\circ\widetilde\phi_r}&\DCA\ar[u]_{\overline{\pi}_r\circ\phi_r}}
\end{equation}
is commutative. Then the maps $\pi_r$, $\widetilde\pi_r\circ\widetilde\phi_r$ and $\overline\pi_r\circ\phi_r$ are bounded anti-homomorphisms, $\delta$ and $i$ are bounded homomorphisms, and $\overline\pi_r\circ\phi_r$ is unital. One has $\Vert\delta\Vert\leq 1$, $\Vert i\Vert=1$ if $\delta(A)\neq 0$, $\Vert\widetilde\pi_r\circ\widetilde\phi_r\Vert\leq m\Vert\pi_r\Vert$, and $\Vert\overline\pi_r\circ\phi_r\Vert\leq m\Vert\pi_r\Vert$. In particular, $X$ becomes a non-degenerate normed right $\DCA$-module.

The image $\pi_r(A)$ is a two-sided ideal in $(\overline\pi_r\circ\phi_r)(\DCA)$. In fact, if $(L,R)\in\DCA$ and $a\in A$, then
$(\overline\pi_r\circ\phi_r)((L,R))\pi_r(a)=\overline\pi_r(R)\pi_r(a)=\pi_r(R(a))$, and $\pi_r(a)(\overline\pi_r\circ\phi_r)((L,R))=\pi_r(a)\overline\pi_r(R)=\pi_r(L(a))$.

If $(e_i)_{i\in I}$ is any bounded approximate right identity for $A$, and $(L,R)\in\DCA$, then $(\overline\pi_r\circ\phi_r)((L,R))=\textup{SOT-}\lim_i \pi_r(R(e_i))$.

If $A$ is an ordered algebra with a positive bounded approximate right identity, if $X$ is ordered with a closed positive cone, and if $\pi_r$ is positive, then all algebras in the diagram are ordered and all maps are positive.

Alternatively, if, in addition, $A$ has a bounded involution, $X$ is a Hilbert space and $\pi_r$ is involutive, then all algebras in the diagram have a bounded involution and all maps are involutive.

\item If both the left-sided and the right-sided hypotheses apply, then $\overline\pi_l\circ\phi_l$ and $\overline\pi_r\circ\phi_r$ provide $X$ with the structure of a non-degenerate normed $\DCA$-bimodule.

\end{enumerate}
\end{theorem}

\begin{proof}
As to the first part, from the surjectivity of $\delta$ there is at most one diagonal map making the upper triangle commutative, and we see from the discussion surrounding \eqref{diag:leftfirst} that the bounded homomorphism $\widetilde\pi_l\circ\widetilde\phi_l:\delta(A)\to\mathcal B(X)$ has this property. Since $\phi_l$ is contractive, this discussion also shows that $\Vert\widetilde\pi_l\circ\widetilde\phi_l\Vert\leq m\Vert\pi_l\Vert$.

Furthermore, we know from the discussion of centralizer algebras that $\delta(A)$ is a two-sided ideal in $\DCA$, and that $(\delta(e_i))_{i\in I}$ is an ($m$-bounded)(positive) approximate left identity in $\delta(A)$ if $(e_i)_{i\in I}$ is an ($m$-bounded) (positive) approximate left identity in $A$. We are now once more in the situation of the first part of Theorem~\ref{thm:basicthm}, and we conclude that there is at most one homomorphism from $\DCA$ into $\mathcal B(X)$ making the lower triangle commutative, and obviously $\overline\pi_l\circ\phi_l$ meets this requirement. Certainly $\Vert\overline\pi_l\circ\phi_l\Vert\leq m\Vert\pi_l\Vert$.  Furthermore, the strong limit in the first part of Theorem~\ref{thm:basicthm} translates into $(\overline\pi_l\circ\phi_l)((L,R))=\textup{SOT-}\lim_i(\widetilde\pi_l\circ\widetilde\phi_l)((L,R)(\lambda(e_i),\rho(e_i))=\textup{SOT-}\lim_i(\widetilde\pi_l\circ\widetilde\phi_l)( (\lambda(L(e_i)),\rho(R(e_i))))=\textup{SOT-}\lim_i(\widetilde\pi_l\circ\lambda)(L(e_i))=\textup{SOT-}\lim_i \pi_l(L(e_i))$.

All non-involutive statements in the first part are now clear, except the claim that $\pi_l(a)(\overline\pi_l\circ\phi_l)((L,R))=\pi_l(a)\overline\pi_l(L)=\pi_l(R(a))$. The first equality holds by the definition of $\phi_l$. As to the second we compute, for $b\in A$ and $x\in X$, that $\pi_l(a)\overline\pi_l(L)\pi_l(b)x=\pi_l(a)\pi_l(L(b))x=\pi_l(aL(b))x=\pi_l(R(a)b)x=\pi_l(R(a))\pi_l(b)x$. Hence the second equality follows from the non-degeneracy  $X$ as a left $A$-module.

Turning to the involutive case in the first part, we note that $\delta(A)$ is a two-sided ideal of $\DCA$ which is invariant under the involution of $\DCA$. In fact, $(\lambda(a),\rho(a))^*=(\lambda(a^*),\rho(a^*))$. Together with the fact that $\pi_l$ is involutive this implies that $\widetilde\pi_l\circ\phi_l$ is involutive, and then the first part of Theorem~\ref{thm:basicthm} asserts that $\widetilde\pi_l\circ\phi_l$ is involutive. The proof of the first part is now complete.

As to the second part, it is again clear that the diagonal map must be the anti-homomorphism $\widetilde\pi_r\circ\widetilde\phi_r:\delta(A)\to\mathcal B(X)$, and that $\Vert\widetilde\pi_r\circ\widetilde\phi_r\Vert\leq m\Vert\pi_r\Vert$. Thus $X$ becomes a non-degenerate normed right $\delta(A)$-module. Furthermore, we know from the discussion of centralizer algebras that $\delta(A)$ is a two-sided ideal in $\DCA$, and that $(\delta(e_i))_{i\in I}$ is an ($m$-bounded)(positive) approximate right identity in $\delta(A)$ if $(e_i)_{i\in I}$ is an ($m$-bounded) (positive) approximate right identity in $A$. Therefore the second part of Theorem~\ref{thm:basicthm} applies and shows that there is a unique anti-homomorphism from $\DCA$ into $\mathcal B(X)$ making the lower triangle commutative. Obviously $\overline\pi_r\circ\phi_r$ has this property. The rest of the second part is then proved analogously to the first part.

The third part is clear.
\end{proof}

\begin{remark}\label{rem:doubleremark}
As in Remarks~\ref{rem:leftremark} and \ref{rem:rightremark}, $\overline\pi_l\circ\phi_l$ is the unique homomorphism making the square in diagram~\eqref{diag:doubleleft} commutative. Indeed, if $\alpha_l:\DCA\to\mathcal B(X)$ is such a homomorphism, then one sees easily that, for $(L,R)\in\DCA$, $a\in A$, and $x\in X$, one must have $\alpha_l((L,R))\pi_l(a)x=\pi_l(L(a))x$. Likewise, if $\alpha_r:\DCA\to\mathcal B(X)$ is an anti-homomorphism making the square in diagram~\eqref{diag:doubleright} commutative, then it is determined by the requirement that $\alpha_r((L,R))\pi_r(a)x=\pi_r(R(a))x$, for all $(L,R)\in\DCA$, $a\in A$, and $x\in X$.
\end{remark}

\section{Module structures for centralizer algebras: faithful case}\label{sec:modulesovercentralizeralgebrasfaithfulcase}

We will now consider normed modules which are not only non-degenerate, but also faithful. In that case, the associated modules for centralizer algebras are also faithful. If the faithful module is a topological (anti-)embedding of the original algebra, then the same holds for the centralizer algebras, which (anti-)embed as appropriate normalizers of the image of the algebra. The details follow. As a preparation, we show that the mere existence of a non-degenerate faithful module is strongly related to  the injectivity of various maps between the algebra and its centralizer algebras.

\begin{proposition}\label{prop:injectivity} Let $A$ be a normed algebra.
\begin{enumerate}
\item If $A$ has a left approximate identity, then the following are equivalent:
\begin{enumerate}
\item There exist a normed space $X$ and an injective homomorphism $\pi_l: A\to\mathcal B(X)$ providing $X$ with the structure of a non-degenerate faithful normed left $A$-module;
\item In $A\overset{\delta}{\to}\DCA\overset{\phi_l}{\to}\LCA$, the canonical maps $\delta$ and $\phi_l$ are both injective homomorphisms;
\item The canonical map $\lambda:A\to\LCA$ is an injective homomorphism.
\end{enumerate}
\item If $A$ has a right approximate identity, then the following are equivalent:
\begin{enumerate}
\item There exist a normed space $X$ and an injective anti-homomorphism $\pi_r: A\to\mathcal B(X)$ providing $X$ with the structure of a non-degenerate faithful normed right $A$-module;
\item In $A\overset{\delta}{\to}\DCA\overset{\phi_r}{\to}\RCA$, the canonical map $\delta$, resp.\ $\phi_r$, is an injective homomorphism, resp.\ an injective anti-homomorphism;
\item The canonical map $\rho\!:A\!\to\RCA$ is an injective anti-homomorphism.
\end{enumerate}
\end{enumerate}
\end{proposition}

\begin{proof}
We prove only the first part, the second being proved similarly.

Suppose that $X$ is a non-degenerate faithful left $A$-module. Let $(L,R)\in\DCA$ and suppose $\phi_l((L,R))=0$, i.e., $L=0$. Then, for $a,b\in A$, $x\in X$, and $(L,R)\in\DCA$, we have $\pi_l(R(a))\pi_l(b)x=\pi_l(R(a)b)x=\pi_l(aL(b))x$. Hence, if $L=0$, then $\pi_l(R(a))=0$ by non-degeneracy, implying $R(a)=0$ by the injectivity of $\pi_l$. Hence $R=0$, and $\phi_l$ is injective. If $a\in A$ and $\delta(a)=0$, then certainly $\lambda(a)=0$. In that case, for $a,b\in A$, and $x\in X$, we have $\pi_l(a)\pi_l(b)x=\pi_l(\lambda(a)b)x=0$. By non-degeneracy, $\pi_l(a)=0$, so that $a=0$ by the injectivity of $\pi_l$. This shows that (a) implies (b).

Since $\lambda=\phi_l\circ\delta$, it is trivial that (b) implies (c).

Assuming (c), it is sufficient to take $X=A$ and $\pi_l=\lambda: A\to\mathcal B(A)$. The module is faithful by assumption, and non-degenerate by the existence of a left approximate identity.
\end{proof}

\begin{remark}
Note that is not assumed that the ones-sided approximate identity is bounded. Also, it is remarkable how little is needed to show that (a) implies (b) (and hence trivially also (c)). In the left-sided case, if $A$ is an abstract algebra, $X$ is a topological vector space, and $\pi_l:A\to\mathcal B(X)$ is an injective homomorphism from $A$ into the continuous linear maps from $X$ into itself such that the elements $\pi_l(a)x$, for $a\in A$ and $x\in X$, span a dense subspace of $X$, then, by the same proof, $\delta$ and $\phi_l$ are injective, where $\DCA$ and $\LCA$ are then defined purely algebraically.
\end{remark}

Before stating the left-sided version of the main result on embedding of centralizer algebras, we introduce the necessary notation. If $X$ is a normed space, and $\mathcal S\subset\mathcal B(X)$, then let $N_l(\mathcal S, \mathcal B(X))=\{T\in\mathcal B(X) : TS\in\mathcal S\textup{ for all }S\in\mathcal S\}$, $N_r(\mathcal S, \mathcal B(X))=\{T\in\mathcal B(X) : ST\in\mathcal S\textup{ for all }S\in\mathcal S\}$, and $N(\mathcal S,\mathcal B(X))=N_l(\mathcal S, \mathcal B(X))\cap N_r(\mathcal S, \mathcal B(X))$. If $\mathcal S$ is a subalgebra of $\mathcal B(X)$, then $N_l(\mathcal S, \mathcal B(X))$, resp.\,$N_r(\mathcal S, \mathcal B(X))$, resp.\,$N(\mathcal S,\mathcal B(X))$ carries two natural norms: the norm from $\mathcal B(X)$ and the norm from $\mathcal M_l(\mathcal S)$, resp.\ $\mathcal M_r(\mathcal S)$, resp.\ $\mathcal M(\mathcal S)$.

Remarkably enough, if $X$ is a normed space, $A$ is a normed algebra with a bounded left, resp.\ right, approximate identity, and if $\pi: A\to\mathcal B(X)$ provides $X$ with the structure of a non-degenerate left, resp.\ right, $A$-module, then, in both the left-sided and right-sided case, the two canonical norms on each of $N_l(\pi(A),\mathcal B(X))$, $N_r(\pi(A),\mathcal B(X))$, and $N(\pi(A),\mathcal B(X))$, are equivalent. Introducing some notation to make this precise, if $T\in N_l(\pi(A),\mathcal B(X))$, we let $\Vert\lambda(T)\Vert$ denote the norm of the left multiplication with $T$ as an element of $\mathcal M_l(\pi(A))$. Clearly $\Vert\lambda(T)\Vert\leq\Vert T\Vert$. Similarly, if $T\in N_r(\pi(A),\mathcal B(X))$, we let $\Vert\rho(T)\Vert$ denote the norm of the right multiplication with $T$ as an element of $\mathcal M_r(\pi(A))$, and clearly $\Vert\rho(T)\Vert\leq\Vert T\Vert$. If $T\in N(\pi(A),\mathcal B(X))$, let $\Vert\delta(T)\Vert=\max (\Vert\lambda(T)\Vert, \Vert\rho(T)\Vert)$, so that $\Vert\delta(T)\Vert\leq\Vert T\Vert$. To see the equivalences, let $(e_i)_{i\in I}$ be an $m$-bounded left, resp.\ right, approximate identity for $A$. As observed in Remark~\ref{rem:nondegeneratesubmodule}, $\textup{SOT}-\lim_i\pi(e_i)=\textup{id}_X$ in both the left-sided and right-sided case. Hence, if $T\in N_l(\pi(A),\mathcal B(X))$, and $x\in X$, then $Tx=\lim_i T\pi(e_i)x$. Since $\Vert T\pi(e_i)\Vert\leq\Vert\lambda(T)\Vert\Vert\pi(e_i)\Vert$, we see that $\Vert T\Vert\leq m \Vert\pi\Vert \Vert\lambda(T)\Vert$. Likewise, $T\in N_r(\pi(A),\mathcal B(X))$, and $x\in X$, then $Tx=\lim_i \pi(e_i)Tx$, which implies that $\Vert T\Vert\leq m \Vert\pi\Vert\Vert\rho(T)\Vert$. Finally, if $T\in N(\pi(A),\mathcal B(X))$, then clearly $\Vert T\Vert\leq m \Vert\delta(T)\Vert\Vert\pi\Vert$.

\begin{theorem}\label{thm:leftembeddings}
Let $A$ be a normed algebra with an $m$-bounded approximate left identity, and let $X$ be a Banach space. Suppose that $\pi_l:A\to\mathcal B(X)$ provides $X$ with the structure of a non-degenerate faithful normed left $A$-module. Then all maps in the diagram
\begin{equation}
\xymatrix
{A\ar@{^{(}->}[r]^{\pi_l}\ar[d]_{\lambda}& \mathcal B(X)\\
\lambda(A)\ar@{^{(}->}[r]_i&\LCA\ar@{^{(}->}[u]_{\overline{\pi}_l}}
\end{equation}
from Theorem~\ref{thm:leftcentralizers} are injective homomorphisms. The canonical homomorphism $\phi_l:\DCA\to\LCA$ is also injective.

Suppose, for the remainder of this Theorem, that $\pi_l: A\to\pi_l(A)$ has a bounded inverse $\pi_l^{-1}:\pi_l(A)\to A$.

If $N_l(\pi_l(A),\mathcal B(X))$ carries the norm from either $\mathcal B(X)$ or $\mathcal M_l(\pi_l(A))$, then $\overline\pi_l$ is a bounded algebra isomorphism between $\LCA$ and $N_l(\pi_l(A),\mathcal B(X))$, with $\Vert\overline\pi_l\Vert\leq m\Vert\pi_l\Vert$ in both cases, and the inverse map $\overline\pi^{-1}_l:N_l(\pi_l(A),\mathcal B(X))\to\LCA$ is also bounded, with $\Vert\overline\pi^{-1}_l\Vert\leq\Vert\pi_l\Vert\Vert\pi_l^{-1}\Vert$ in both cases.

Likewise, if $N(\pi_l(A),\mathcal B(X))$ carries the norm from either $\mathcal B(X)$ or $\mathcal M(\pi_l(A))$, then $\overline\pi_l\circ\phi_l:\DCA\to\mathcal B(X)$ yields a bounded algebra isomorphism between $\DCA$ and $N(\pi_l(A),\mathcal B(X))$, with $\Vert\overline\pi_l\circ\phi_l\Vert\leq m \Vert\pi_l\Vert$ in both cases, and the inverse $(\overline\pi_l\circ\phi_l)^{-1}:N(\pi_l(A),\mathcal B(X))\to\DCA$ is also bounded, with $\Vert(\overline\pi_l\circ\phi_l)^{-1}\Vert\leq\Vert\pi_l\Vert\Vert\pi_l^{-1}\Vert$ in both cases.

If, in addition, $A$ is an ordered normed algebra with a positive bounded approximate left identity, $X$ is ordered with a closed positive cone, and $\pi_l$ is an isomorphism of ordered algebras between $A$ and $\pi_l(A)$, then $\overline\pi_l:\LCA\to\mathcal B(X)$ yields an isomorphism of ordered algebras between $\LCA$ and $N_l(\pi_l(A),\mathcal B(X))$, and $\overline\pi_l\circ\phi_l:\DCA\to\mathcal B(X)$ yields an isomorphism of ordered algebras between $\DCA$ and $N(\pi_l(A),\mathcal B(X))$.

Alternatively, if, in addition, $A$ has a bounded involution, $X$ is a Hilbert space and $\pi_l$ is involutive, then $\overline\pi_l\circ\phi_l:\DCA\to\mathcal B(X)$ yields an isomorphism of involutive algebras between $\DCA$ and $N(\pi_l(A),\mathcal (X))$.

\end{theorem}

\begin{proof}
By Proposition~\ref{prop:injectivity}, $\lambda$ and $\phi_l$ are injective. From Theorem~\ref{thm:leftcentralizers} we know that $\overline\pi_l(L)\pi_l(a)=\pi_l(L(a))$, for all $L\in\LCA$ and $a\in A$, so that the injectivity of $\pi_l$ implies that $\overline\pi_l$ is injective.

For the remainder assume, then, that $\pi_l^{-1}: \pi_l(A)\to A$ is bounded.

As observed in Theorem~\ref{thm:leftcentralizers}, it follows from $\overline\pi_l(L)\pi_l(a)=\pi_l(L(a))$, for all $L\in\LCA$ and $a\in A$, that $\overline\pi_l$ maps $\LCA$ into $N_l(\pi_l(A),\mathcal B(X))$. We know from Theorem~\ref{thm:leftcentralizers} that $\Vert\overline\pi_l\Vert\leq m\Vert\pi_l\Vert$ if $N_l(\pi_l(A),\mathcal B(X))$ carries the norm from $\mathcal B(X)$, and hence this upper bound also holds when it carries the norm from $\mathcal M_l(\pi_l(A))$. To construct the inverse of $\pi_l$ we define, in an anticipating notation, the map $\overline\pi_l^{-1}: N_l(\pi_l(A),\mathcal B(X))\to\mathcal B(A)$ as $\overline\pi_l^{-1}(T)(a)=\pi_l^{-1}(T\pi_l(a))$, for $T\in N_l(\pi_l(A),\mathcal B(X))$ and $a\in A$. Clearly, for both norms on $N_l(\pi_l(A))$, $\overline\pi^{-1}_l$ is a bounded homomorphism, and $\Vert\overline\pi^{-1}_l\Vert\leq\Vert\pi_l\Vert\Vert\pi_l^{-1}\Vert$.
For $a,b\in A$, and $T\in N_l(\pi_l(A),\mathcal B(X))$, we have $\overline\pi^{-1}_l(T)(ab)=\pi_l^{-1}(T\pi_l(ab))=\pi_l^{-1}((T\pi_l(a))\pi_l(b))=\pi_l^{-1}(T\pi_l(a))b=\overline\pi_l^{-1}(T)(a)b$, so that $\overline\pi_l^{-1}(T)$ is a left centralizer. Hence, in fact, $\overline\pi_l^{-1}$ maps $N_l(\pi_l(A),\mathcal B(X))$ into $\LCA$.
For $T\in N_l(\pi_l(A)$, $\mathcal B(X)),a\in A$, and $x\in X$, we have $(\overline\pi_l\circ\overline\pi_l^{-1}(T))\pi_l(a)x=\overline\pi_l(\overline\pi_l^{-1}(T))\overline\pi_l(\lambda(a))x=\overline\pi_l(\overline\pi_l^{-1}(T)\circ\lambda(a))x=\overline\pi_l(\lambda(\overline\pi_l^{-1}(T)(a)))x= \pi_l(\overline\pi_l^{-1}(T)a)x=T\pi_l(a)x$. Since $X$ is non-degenerate, we conclude that $\overline\pi_l\circ\overline\pi_l^{-1}(T)=T$.
Furthermore, if $L\in\LCA$, and $a\in A$, then $(\overline\pi_l^{-1}\circ\overline\pi_l)(L)(a)=\pi_l^{-1}(\overline\pi_l(L)\pi_l(a))=\pi_l^{-1}(\pi_l(L(a)))=L(a)$. Hence $(\overline\pi_l^{-1}\circ\overline\pi_l)(L)=L$.
This concludes the proof of the statement that $\overline\pi_l$ is a topological embedding of $\LCA$ in $\mathcal B(X)$ and of the upper bound for the norm of its inverse, for both norms on $N_l(\pi_l(A))$.

Now we turn to the statements on $\DCA$. We have already observed in the first part of Theorem~\ref{thm:doublecentralizers} that $\overline\pi_l\circ\phi_l$ maps $\DCA$ into $N(\pi_l(A),\mathcal B(X))$. To construct its inverse, we define the additional map $\mu_l: N_r(\pi_l(A),\mathcal B(X))\to\mathcal B(A)$ as $\mu_l(T)(a)=\pi_l^{-1}(\pi_l(a)T)$, for $T\in N_r(\pi_l(A),\mathcal B(X))$, and $a\in A$. Clearly $\mu_l$ is a bounded anti-homomorphism, and $\Vert\mu_l\Vert\leq\Vert\pi_l\Vert\Vert\pi_l^{-1}\Vert$, for both norms on $N_r(\pi_r(A),\mathcal B(X))$.
For $a,b\in A$ and $T\in N_r(\pi_l(A),\mathcal B(X))$ we have $\mu_l(T)(ab)=\pi_l^{-1}(\pi_l(ab)T)=\pi_l^{-1}(\pi_l(a)(\pi_l(b)T))=a\mu_l(T)(b)$. Hence $\mu_l(T)$ is a right centralizer and we have, in fact, a bounded anti-homomorphism $\mu_l:N_r(\pi_l(A),\mathcal B(X))\to\RCA$.

Suppose now that $T\in N(\pi_l(A),\mathcal B(X))$. Then the pair $(\overline\pi_l^{-1}(T), \mu_l(T))$ is a double centralizer. Indeed, we already know that $\overline\pi_l^{-1}(T)\in\LCA$ and that $\mu_l(T)\in\RCA$, and furthermore, for $a,b\in A$, we have $a(\overline\pi_l^{-1}(T)(b))=a\pi_l^{-1}(T\pi_l(b))=\pi_l^{-1}(\pi_l(a)T\pi_l(b))=\pi_l^{-1}(\pi_l(a)T)b=(\mu_l(T)(a))b$. Thus we obtain a map $\psi_l:N(\pi_l(A),\mathcal B(X))\to\DCA$ which is defined, for $T\in N(\pi_l(A),\mathcal B(X))$, by $\psi_l(T)=(\overline\pi_l^{-1}(T),\mu_l(T))$. Clearly $\psi_l$ is a bounded homomorphism, and $\Vert\psi_l\Vert\leq\Vert\pi_l\Vert\Vert\pi_l^{-1}\Vert$ for both norms on $N(\pi_l(A),\mathcal B(X))$, since both $\overline\pi_l^{-1}$ and $\mu_l$ satisfy this estimate in two cases.

We proceed by showing that $\psi_l:N(\pi_l(A),\mathcal B(X))\to\DCA$ and $\overline\pi_l\circ\phi_l:\DCA\to N(\pi_l(A),\mathcal B(X))$ are inverse to each other. It is immediate from the definitions that $(\overline\pi_l\circ\phi_l)\circ\psi_l$ is the identity on $N(\pi_l(A),\mathcal B(X))$. In the other direction, let $(L,R)\in\DCA$. Then $(\psi_l\circ(\overline\pi_l\circ\phi_l))((L,R))=(L,\mu_l(\overline\pi_l(L)))$. Now, for $a\in A$,
$\mu_l(\overline\pi_l(L))(a)=\pi_l^{-1}(\pi_l(a)\overline\pi_l(L))$. Hence $\pi_l(\mu_l(\overline\pi_l(L))(a))=\pi_l(a)\overline\pi_l(L)$. On the other hand, we had already observed in the first part of Theorem~\ref{thm:doublecentralizers} that $\pi_l(a)\overline\pi_l(L)=\pi_l(R(a))$. By the injectivity of $\pi_l$ we conclude that $\mu_l(\overline\pi_l(L))(a)=R(a)$, and hence $\psi_l\circ(\overline\pi_l\circ\phi_l)$ is the identity on $\DCA$. This concludes the proof of the statements concerning $\DCA$.

We now turn to the ordered situation. We know from Theorem~\ref{thm:leftcentralizers} that $\overline\pi_l$ is positive. Since we have assumed that $\pi_l^{-1}:\pi_l(A)\to A$ is positive, it is immediate from $\overline\pi_l^{-1}(T)(a)=\pi_l^{-1}(T\pi_l(a))$, for $T\in N_l(\pi_l(A),\mathcal B(X))$, that $\overline\pi_l^{-1}$ is positive. Hence $\overline\pi_l:N_l(\pi_l(A))\to\LCA$ is an isomorphism of ordered algebras. Likewise, we know from the first part of Theorem~\ref{thm:doublecentralizers} that $\overline\pi_l\circ\phi_l$ is positive. Since $\mu_l(T)(a)=\pi_l^{-1}(\pi_l(a)T)$, for $T\in N_r(\pi_l(A),\mathcal B(X))$, the assumption that $\pi_l^{-1}$ is positive shows that $\mu_l$ is positive. Hence this is also true for $\psi_l$, and $\overline\pi_l\circ\phi_l:\DCA\to N(\pi_l(A),\mathcal B(X))$ is an isomorphism of ordered algebras.

As to the involutive situation, we know from the first part of Theorem~\ref{thm:doublecentralizers} that $\overline\pi_l\circ\phi_l$ is involutive. Hence so is the image, and since the inverse of an involutive map is necessarily involutive, we are done.
\end{proof}

Suppose that, in Theorem~\ref{thm:leftembeddings}, $\pi_l: A\to\pi_l(A)$ has a bounded inverse $\pi_l^{-1}:\pi_l(A)\to A$. Then the definition of $\overline\pi_l^{-1}$ in the proof shows that, after identifying $A$ with its image $\pi_l(A)$, $\LCA$ is to be identified with all left multiplications by elements of $N_l(\pi_l(A),\mathcal B(X))$, and, likewise, $\DCA$ is to be identified with all pairs consisting of a left and a right multiplication by the same element of $N(\pi_l(A),\mathcal B(X))$. Therefore the following excerpts from Theorem~\ref{thm:leftembeddings} hold.

\begin{corollary}\label{cor:leftisomorphiccopies}
Let $A$ be a normed algebra with a bounded approximate left identity, and let $X$ be a Banach space. Suppose that $\pi_l:A\to\mathcal B(X)$ provides $X$ with the structure of a non-degenerate faithful normed left $A$-module, and that $\pi_l$ is an embedding of $A$ as a topological algebra. Then $\LCA$ is canonically isomorphic, as a topological algebra, with $N_l(\pi_l(A),\mathcal B(X))$, where $N_l(\pi_l(A),\mathcal B(X))$ can carry either the norm from $\mathcal B(X)$ or the equivalent norm from $\mathcal M_l(\pi_l(A))$. Likewise, $\DCA$ is canonically isomorphic, as a topological algebra, with $N(\pi_l(A),\mathcal B(X))$, where $N(\pi_l(A),\mathcal B(X))$ can carry either the norm from $\mathcal B(X)$ or the equivalent norm from $\mathcal M(\pi_l(A))$.

If, in addition, $A$ is an ordered normed algebra with a positive bounded approximate left identity, $X$ is ordered with a closed positive cone, and $\pi_l$ is an isomorphism of ordered algebras between $A$ and $\pi_l(A)$, then the above two canonical isomorphisms are isomorphisms of ordered algebras.

Alternatively, if, in addition, $A$ has a bounded involution, $X$ is a Hilbert space, and $\pi_l$ is involutive, then the above canonical isomorphism between $\DCA$ and $N(\pi_l(A),\mathcal B(X))$ is an isomorphism of involutive algebras.
\end{corollary}

\begin{corollary}\label{cor:leftisometriccopies}
Let $X$ be a Banach space, and suppose that $A$ is a (not necessarily closed) subalgebra of $\mathcal B(X)$, acting non-degenerately on $X$. If $A$ has a bounded approximate left identity, then $\LCA$, resp.\ $\DCA$, is canonically isometrically isomorphic with $N_l(A,\mathcal B(X))$, resp.\ $N(A,\mathcal B(X))$, where the latter algebra is supplied with the norm from $\mathcal M_l(A)$, resp.\ $\mathcal M(A)$.

If, in addition, $X$ is ordered with a closed positive cone, and $A$ has a positive bounded approximate left identity, then the above two canonical isomorphisms are isomorphisms of ordered algebras.

Alternatively, if, in addition, $X$ is a Hilbert space, and $A$ is involutive, then the canonical isomorphism between $\mathcal M(A)$ and $N(A,\mathcal B(X))$ is involutive.
\end{corollary}

\begin{remark}\label{rem:isometricembedding}
If $A$ is a normed algebra with a bounded left approximate identity, then it follows from Corollary~\ref{cor:leftisometriccopies} that, for each isometric non-degenerate embedding into $\mathcal B(X)$, for some Banach space $X$, the left normalizer of the image, resp.\ the normalizer of the image, is always canonically isometrically isomorphic with $\mathcal M_l(A)$, resp.\ $\mathcal M(A)$. If, in addition, $A$ is involutive with an isometric involution, then, for each involutive isometric non-degenerate embedding into $\mathcal B(X)$, for some Hilbert space $X$, the normalizer of the image is always isometrically involutively isomorphic with $\mathcal M(A)$. As a special case, we retrieve the well known fact that the multiplier algebra of a $C^*$-algebra is $C^*$-isomorphic with the normalizer of the image in any faithful non-degenerate involutive Hilbert representation.
\end{remark}

\begin{remark}\label{rem:compactoperators}
If $X$ is a Banach space, let $\mathcal K(X)$ denote the compact operators on $X$. For $T\in\mathcal B(X)$, it is easily checked that the norm of the corresponding left and right multiplication on $\mathcal K (X)$ is in both cases equal to the norm of $T$ as an element of $\mathcal B(X)$. If $\mathcal K(X)$ has a bounded left approximate identity (e.g., if $X$ has a Schauder basis), then Corollary~\ref{cor:leftisometriccopies} therefore asserts that $\mathcal M_l(\mathcal K(X))$ and $\mathcal M(\mathcal K(X))$ are both canonically isometrically isomorphic with $\mathcal B(X)$. More is true, however: these isometric isomorphisms are both valid without any assumption on $\mathcal K(X)$, see \cite[p.\ 313]{JohnsonLMS}. In the same paper, it is also shown \cite[p.\ 314]{JohnsonLMS} that $\mathcal M_r(\mathcal K(X))$ is always isometrically isomorphic with the algebra of bounded operators on the dual space of $X$. We refer to \cite[Section~1.7.14]{PalmerI} for further results in this vein.
\end{remark}

The right-sided version of Theorem~\ref{thm:leftembeddings} reads as follows. The statement that the right centralizer algebra of $A$ is isomorphic with the left normalizer of the image is not a mistake.

\begin{theorem}\label{thm:rightembeddings}
Let $A$ be a normed algebra with an $m$-bounded approximate right identity, and let $X$ be a Banach space. Suppose that $\pi_r:A\to\mathcal B(X)$ provides $X$ with the structure of a non-degenerate faithful normed right $A$-module. Then in the diagram
\begin{equation}
\xymatrix
{A\ar@{^{(}->}[r]^{\pi_r}\ar[d]_{\rho}& \mathcal B(X)\\
\rho(A)\ar@{^{(}->}[r]_i&\RCA\ar@{^{(}->}[u]_{\overline{\pi}_r}}
\end{equation}
from Theorem~\ref{thm:rightcentralizers}, $\pi_r$ and $\rho$ are injective anti-homomorphisms, and $i$ and $\overline\pi_r$ are injective homomorphisms. The canonical anti-homomorphism $\phi_r:\DCA\to\RCA$ is also injective.

Suppose, for the remainder of this Theorem, that $\pi_r: A\to\pi_r(A)$ has a bounded inverse $\pi_r^{-1}:\pi_r(A)\to A$.

If $N_l(\pi_r(A),\mathcal B(X))$ carries the norm from either $\mathcal B(X)$ or $\mathcal M_l(\pi_r(A))$, then $\overline\pi_r$ is a bounded algebra isomorphism between $\RCA$ and $N_l(\pi_r(A),\mathcal B(X))$, with $\Vert\overline\pi_r\Vert\leq m\Vert\pi_r\Vert$ in both cases, and the inverse map $\overline\pi^{-1}_r:N_l(\pi_r(A),\mathcal B(X))\to\RCA$ is also bounded, with $\Vert\overline\pi^{-1}_r\Vert\leq\Vert\pi_r\Vert\Vert\pi_r^{-1}\Vert$ in both cases.

Likewise, if $N(\pi_r(A),\mathcal B(X))$ carries the norm from either $\mathcal B(X)$ or $\mathcal M(\pi_r(A))$, then $\overline\pi_r\circ\phi_r:\DCA\to\mathcal B(X)$ yields a bounded algebra anti-isomorphism between $\DCA$ and $N(\pi_r(A),\mathcal B(X))$, with $\Vert\overline\pi_r\circ\phi_r\Vert\leq m \Vert\pi_r\Vert$ in both cases, and the inverse map $(\overline\pi_r\circ\phi_r)^{-1}:N(\pi_r(A),\mathcal B(X))\to\DCA$ is also bounded, with $\Vert(\overline\pi_r\circ\phi_r)^{-1}\Vert\leq\Vert\pi_r\Vert\Vert\pi_r^{-1}\Vert$ in both cases.

If, in addition, $A$ is an ordered normed algebra with a positive bounded approximate right identity, $X$ is ordered with a closed positive cone, and $\pi_r$ is an anti-isomorphism of ordered algebras between $A$ and $\pi_r(A)$, then $\overline\pi_r:\RCA\to\mathcal B(X)$ yields an isomorphism of ordered algebras between $\RCA$ and $N_l(\pi_r(A),\mathcal B(X))$, and $\overline\pi_r\circ\phi_r:\DCA\to\mathcal B(X)$ yields an anti-isomorphism of ordered algebras between $\DCA$ and $N(\pi_r(A),\mathcal B(X))$.

Alternatively, if, in addition, $A$ has a bounded involution, $X$ is a Hilbert space and $\pi_r$ is involutive, then $\overline\pi_r\circ\phi_r:\DCA\to\mathcal B(X)$ yields an anti-isomorphism of involutive algebras between $\DCA$ and $N(\pi_r(A),\mathcal (X))$.

\end{theorem}

\begin{proof}
The proof is similar to that of Theorem~\ref{thm:leftembeddings}, and uses Theorem~\ref{thm:rightcentralizers} and the second part of Theorem~\ref{thm:doublecentralizers}. The first step is to prove that the homomorphism $\overline\pi_r^{-1}: N_l(\pi_r(A),\mathcal B(X))\to\RCA$, defined by $\overline\pi_r^{-1}(T)(a)=\pi_r^{-1}(T\pi_r(a))$, for $T\in N_l(\pi_r(A),\mathcal B(X))$ and $a\in A$, is the two-sided inverse of $\overline\pi_r:\RCA\to N_l(\pi_r(A),\mathcal B(X))$. The role of the anti-homomorphism $\mu_l$ in the previous proof is taken over by the anti-homomorphism $\mu_r: N_r(\pi_r(A),\mathcal B(X))\to\LCA$, defined by $\mu_r(T)(a)=\pi_r^{-1}(\pi_r(a)T)$, for $T\in N_r(\pi_r(A),\mathcal B(X))$, and $a\in A$. These combine to the anti-homomorphism $(\overline\pi_r\circ\phi)^{-1}: N(\pi_l(A), \mathcal B(X))\to\DCA$, which is given by $(\overline\pi_r\circ\phi)^{-1}(T)=(\mu_r(T),\overline\pi_r^{-1}(T))$, for $T\in N(\pi_r(A),\mathcal B(X))$.
\end{proof}

\begin{corollary}\label{cor:rightisomorphiccopies}
Let $A$ be a normed algebra with a bounded approximate right identity, and let $X$ be a Banach space. Suppose that $\pi_r:A\to\mathcal B(X)$ provides $X$ with the structure of a non-degenerate faithful normed right $A$-module, and that $\pi_r$ is an anti-embedding of $A$ as a topological algebra. Then $\RCA$ is canonically isomorphic, as a topological algebra, with $N_l(\pi_r(A),\mathcal B(X))$, where $N_l(\pi_r(A),\mathcal B(X))$ can carry either the norm from $\mathcal B(X)$ or the equivalent norm from $\mathcal M_l(\pi_r(A))$. Likewise, $\DCA$ is canonically anti-isomorphic, as a topological algebra, with $N(\pi_r(A),\mathcal B(X))$, where $N(\pi_r(A),\mathcal B(X))$ can carry either the norm from $\mathcal B(X)$ or the equivalent norm from $\mathcal M(\pi_r(A))$.

If, in addition, $A$ is an ordered normed algebra with a positive bounded approximate right identity, $X$ is ordered with a closed positive cone, and $\pi_r$ is an anti-isomorphism of ordered algebras between $A$ and $\pi_r(A)$, then the above canonical isomorphism between $\RCA$ and $N_l(\pi_r(A),\mathcal B(X))$ is an isomorphisms of ordered algebras, and the above canonical anti-isomorphism between $\DCA$ and $N(\pi_r(A),\mathcal B(X))$ is an anti-isomorphism of ordered algebras.

Alternatively, if, in addition, $A$ has a bounded involution, $X$ is a Hilbert space, and $\pi_r$ is involutive, then the above canonical anti-isomorphism between $\DCA$ and $N(\pi_r(A),\mathcal B(X))$ is an anti-isomorphism of involutive algebras.
\end{corollary}

%%%%%%%%%%%%%%   End of the text   %%%%%%%%%%%%%%%%%%%%%%%%%%%%%%%%%%%%%%%%%%%%%%%%%%%%

%%%%%%%%%%%%%%   Bibliography   %%%%%%%%%%%%%%%%%%%%%%%%%%%%%%%%%%%%%%%%%%%%%%%%%%%%%%%

%%%%%%%%%%%%%%%%%%%%%%%%%%%%%%%%%%%%%%%%%%%%%%%%%%%%%%%%%%%%%%%%%%%%%%%%%%%%%%%%%%%%%%%

%%%%%%%%%%%%%%   End of the document   %%%%%%%%%%%%%%%%%%%%%%%%%%%%%%%%%%%%%%%%%%%%%%%%

\begin{thebibliography}{99}


\bibitem{Blackadar} B.\ Blackadar, \emph{Operator Algebras. Theory of $C^\ast$- Algebras and Von Neumann Algebras}, Encyclopaedia of Mathematical Sciences, Vol.\ 122, Operator Algebras and Noncommutative Geometry, Vol.\ III, Springer-Verlag, Berlin, 2006.

\bibitem{BonsallDuncan} F.F.\ Bonsall, J.\ Duncan, \emph{Complete normed algebras}. Ergebnisse der Mathematik und ihrer Grenzgebiete, Band 80. Springer-Verlag, New York-Heidelberg, 1973.

\bibitem{Dales} H.G.\ Dales, \emph{Banach Algebras and Automatic Continuity}, London Math.\ Soc.\ Monographs 24, Oxford Univ.\ Press, New York, 2000.

\bibitem{Davidson} K.R.\ Davidson, \emph{$C^\ast$-algebras by Example}, Fields Institute Monographs Vol.\ 6,  American Mathematical Society, Providence, RI, 1996.

\bibitem{DidJWo} S.\ Dirksen, M.\ de Jeu, M.\ Wortel, \emph{Beyond the stars: crossed products of Banach algebras}, to appear.

\bibitem{Dixmier} J.\ Dixmier, \emph{$C\sp*$-algebras}, North-Holland Mathematical Library, Vol.\ 15, North-Holland Publishing Co., Amsterdam-New York-Oxford, 1977.

\bibitem{Jameson} G.J.O.\ Jameson, \emph{Ordered Linear Spaces}, Lecture Notes in Math., Vol.\ 141, Springer-Verlag, Heidelberg, 1970.

\bibitem{JohnsonLMS} B.E.\ Johnson, \emph{An introduction to the theory of centralizers}, Proc.\ London Math.\ Soc.\ (3) {\bf 14} (1964), 299-320.

\bibitem{JohnsonMemoir} B.E.\ Johnson, \emph{Cohomology in Banach algebras}, Mem.\ Amer.\ Math.\ Soc.\ 127, American Mathematical Society, Providence, RI, 1972.

\bibitem{Murphy} G.J.\ Murphy, \emph{$C\sp *$-algebras and operator theory}, Academic Press, Boston, MA, 1990.

\bibitem{PalmerI} T.W.\ Palmer, \emph{Banach algebras and the general theory of $\sp *$-algebras, Vol.\ 1: Algebras and Banach algebras}, Encyclopedia of Mathematics and its Applications, 49. Cambridge University Press, Cambridge, 1994.

\bibitem{Williams} D.P.\ Williams, \emph{Crossed products of $C{\sp \ast}$-algebras}, Mathematical Surveys and Monographs, 134, American Mathematical Society, Providence, RI, 2007.


\end{thebibliography}
\end{document}